\documentclass[12pt,reqno,openbib,runningheads,a4paper]{amsart}
\usepackage{eurosym}
\usepackage{amssymb}
\usepackage{amsfonts}
\usepackage{amsmath}
\usepackage{graphicx}
\usepackage[authoryear]{natbib}
\usepackage[dvipsnames]{xcolor}
\usepackage[legalpaper,bookmarks=true,colorlinks=true,linkcolor=blue,citecolor=blue]{hyperref}

\providecommand{\U}[1]{\protect\rule{.1in}{.1in}}
\newtheorem{theorem}{Theorem}[section]

\newtheorem{lemma}{Lemma}[section]

\newtheorem{remark}{Remark}[section]

\makeatletter
\renewcommand{\@biblabel}[1]{}
\makeatother
\input{tcilatex}

\begin{document}

\begin{center}
{\Large \textbf{Kernel estimation of the tail index of a right-truncated P%
\textbf{areto-type distribution}}}

\medskip \medskip

{\large Souad Benchaira, Djamel Meraghni, Abdelhakim Necir}$^{\ast }$\medskip

{\small \textit{Laboratory of Applied Mathematics, Mohamed Khider
University, Biskra, Algeria}}\medskip \medskip 
\begin{equation*}
\end{equation*}
\end{center}

\noindent\textbf{Abstract}\medskip

\noindent In this paper, we define a kernel estimator for the tail index of
a Pareto-type distribution under random right-truncation and establish its
asymptotic normality. A simulation study shows that, compared to the
estimators recently proposed by \cite{GS2015} and \cite{BchMN-15b}, this
newly introduced estimator behaves better, in terms of bias and mean squared
error, for small samples.\medskip

\medskip


\vfill

\vfill

\noindent {\small $^{\text{*}}$Corresponding author: \texttt{%
necirabdelhakim@yahoo.fr} \newline
}

\noindent {\small \textit{E-mail addresses:}\newline
\texttt{benchaira.s@hotmail.fr} (S.~Benchaira)\newline
\texttt{djmeraghni@yahoo.com} (D.~Meraghni)}

\section{\textbf{Introduction\label{sec1}}}

\noindent Let $\left( \mathbf{X}_{i},\mathbf{Y}_{i}\right) ,$ $1\leq i\leq N$
be a sample of size $N\geq 1$ from a couple $\left( \mathbf{X},\mathbf{Y}%
\right) $ of independent random variables (rv's) defined over some
probability space $\left( \Omega ,\mathcal{A},\mathbf{P}\right) ,$ with
continuous marginal distribution functions (df's) $\mathbf{F}$ and $\mathbf{G%
}$ respectively.$\ $Suppose that $\mathbf{X}$ is truncated to the right by $%
\mathbf{Y},$ in the sense that $\mathbf{X}_{i}$ is only observed when $%
\mathbf{X}_{i}\leq \mathbf{Y}_{i}.$ We assume that both survival functions $%
\overline{\mathbf{F}}:=1-\mathbf{F}$\textbf{\ }and $\overline{\mathbf{G}}:=1-%
\mathbf{G}$ are regularly varying at infinity with tail indices $\gamma
_{1}>0$ and $\gamma _{2}>0$ respectively. That is, for any $x>0,$%
\begin{equation}
\lim_{z\rightarrow \infty }\frac{\overline{\mathbf{F}}\left( xz\right) }{%
\overline{\mathbf{F}}\left( z\right) }=x^{-1/\gamma _{1}}\text{ and }%
\lim_{z\rightarrow \infty }\frac{\overline{\mathbf{G}}\left( xz\right) }{%
\overline{\mathbf{G}}\left( z\right) }=x^{-1/\gamma _{2}}.  \label{RV-1}
\end{equation}%
This class of distributions, which includes models such as Pareto, Burr, Fr%
\'{e}chet, stable and log-gamma, plays a prominent role in extreme value
theory. Also known as heavy-tailed, Pareto-type or Pareto-like
distributions, these models have important practical applications and are
used rather systematically in certain branches of non-life insurance as well
as in finance, telecommunications, geology, and many other fields %
\citep[see, e.g.,][]{Res06}. Let us denote $\left( X_{i},Y_{i}\right) ,$ $%
i=1,...,n$ to be the observed data, as copies of a couple of rv's $\left(
X,Y\right) ,$ corresponding to the truncated sample $\left( \mathbf{X}_{i},%
\mathbf{Y}_{i}\right) ,$ $i=1,...,N,$ where $n=n_{N}$ is a sequence of
discrete rv's for which we have,\ by of the weak law of large numbers 
\begin{equation*}
n_{N}/N\overset{\mathbf{p}}{\rightarrow }p:=\mathbf{P}\left( \mathbf{X}\leq 
\mathbf{Y}\right) ,\text{ as }N\rightarrow \infty .
\end{equation*}%
The joint distribution of $X_{i}$ and $Y_{i}$ is%
\begin{equation*}
\begin{array}{ll}
H\left( x,y\right)  & :=\mathbf{P}\left( X\leq x,Y\leq y\right) \medskip  \\ 
& =\mathbf{P}\left( \mathbf{X}\leq \min \left( x,\mathbf{Y}\right) ,\mathbf{Y%
}\leq y\mid \mathbf{X}\leq \mathbf{Y}\right) =p^{-1}\dint_{0}^{y}\mathbf{F}%
\left( x,z\right) d\mathbf{G}\left( z\right) .%
\end{array}%
\end{equation*}%
The marginal distributions of the rv's $X$ and $Y,$ respectively denoted by $%
F$ and $G,$ are equal to $F\left( x\right) =p^{-1}\dint_{0}^{x}\overline{%
\mathbf{G}}\left( z\right) d\mathbf{F}\left( z\right) $ and $G\left(
y\right) =p^{-1}\int_{0}^{y}\mathbf{F}\left( z\right) d\mathbf{G}\left(
z\right) .$ The tail of df $F$ simultaneously depends on $\overline{\mathbf{G%
}}$ and $\overline{\mathbf{F}}$ while that of $\overline{G}$ only relies on $%
\overline{\mathbf{G}}\mathbf{.}$ By using Proposition B.1.10 in \cite{deHF06}%
, to the regularly varying functions $\overline{\mathbf{F}}$ and $\overline{%
\mathbf{G}},$ we show that both $\overline{G}$ and $\overline{F}$ are
regularly varying at infinity as well, with respective indices $\gamma _{2}$
and $\gamma :=\gamma _{1}\gamma _{2}/\left( \gamma _{1}+\gamma _{2}\right) .$
In other words, for any $s>0,$%
\begin{equation}
\lim_{x\rightarrow \infty }\frac{\overline{F}\left( sx\right) }{\overline{F}%
\left( x\right) }=s^{-1/\gamma }\text{ and }\lim_{y\rightarrow \infty }\frac{%
\overline{G}\left( sy\right) }{\overline{G}\left( y\right) }=s^{-1/\gamma
_{2}}.  \label{RV-2}
\end{equation}%
Recently \cite{GS2015} addressed the estimation of the extreme value index $%
\gamma _{1}$ under random right-truncation. They used the definition of $%
\gamma $ to derive the following consistent estimator%
\begin{equation}
\widehat{\gamma }_{1}^{GS}:=k^{-1}\frac{\dsum\limits_{i=1}^{k}\log \dfrac{%
X_{n-i+1:n}}{X_{n-k:n}}\dsum\limits_{i=1}^{k}\log \dfrac{Y_{n-i+1:n}}{%
Y_{n-k:n}}}{\dsum\limits_{i=1}^{k}\log \dfrac{X_{n-k:n}Y_{n-i+1:n}}{%
Y_{n-k:n}X_{n-i+1:n}}},  \label{GS}
\end{equation}%
where $X_{1:n}\leq ...\leq X_{n:n}$ and $Y_{1:n}\leq ...\leq Y_{n:n}$ are
the order statistics pertaining to the samples $\left(
X_{1},...,X_{n}\right) $ and $\left( Y_{1},...,Y_{n}\right) $ respectively
and $k=k_{n}$ is a (random) sequence of discrete rv's satisfying $%
k_{N}\rightarrow \infty $ and $k_{N}/N\rightarrow 0$ as $N\rightarrow \infty
.$ The asymptotic normality of $\widehat{\gamma }_{1}^{GS}$ is established
in \cite{BchMN-15a}, under the tail dependence and the second-order regular
variation conditions. Also, \cite{WW2015} proposed an estimator for $\gamma
_{1}$ and proved its asymptotic normality, by considering a Lyden-Bell
integration with a deterministic threshold. More recently, \cite{BchMN-15b}
treated the case of a random threshold and introduced a Hill-type estimator
for the tail index $\gamma _{1}$ of randomly right-truncated data, as
follows:%
\begin{equation}
\widehat{\gamma }_{1}:=\left( \sum\limits_{i=1}^{k}\frac{\mathbf{F}%
_{n}\left( X_{n-i+1:n}\right) }{C_{n}\left( X_{n-i+1:n}\right) }\right)
^{-1}\sum_{i=1}^{k}\frac{\mathbf{F}_{n}\left( X_{n-i+1:n}\right) }{%
C_{n}\left( X_{n-i+1:n}\right) }\log \frac{X_{n-i+1:n}}{X_{n-k:n}},
\label{BMN}
\end{equation}%
where $\mathbf{F}_{n}\left( x\right) :=\prod_{i:X_{i}>x}\exp \left\{ -\dfrac{%
1}{nC_{n}\left( X_{i}\right) }\right\} ,$ is the well-known Woodroofe's
product-limit estimator \citep{W-85} of the underlying df $\mathbf{F}$ and 
\begin{equation}
C_{n}\left( x\right) :=n^{-1}\sum\limits_{i=1}^{n}\mathbf{1}\left( X_{i}\leq
x\leq Y_{i}\right) .  \label{Cn}
\end{equation}%
The asymptotic normality of $\widehat{\gamma }_{1}$ is established by
considering the second-order regular variation conditions $(\ref%
{second-order})$\ and $(\ref{second-orderG})$ below and the assumption $%
\gamma _{1}<\gamma _{2}.$ The latter condition is required in order to
ensure that it remains enough extreme data for the inference to be accurate.
In other words, we consider the situation where the tail of the rv of
interest $\mathbf{X}$ is not too contaminated by that of the truncating rv $%
\mathbf{Y.}$ Note that, in the presence of complete data, we have $\mathbf{F}%
_{n}\mathbf{\equiv }F_{n}\mathbf{\equiv }C_{n}$ and consequently $\widehat{%
\gamma }_{1}$ reduces to the classical Hill estimator \citep{Hill75}. In
this paper, we derive a kernel version of $\widehat{\gamma }_{1}$ in the
spirit of what is called kernel estimator of \cite{CDM}. Thereby, for a
suitable choice of the kernel function, we obtain an improved estimator of $%
\gamma _{1}$ in terms of bias and mean squared error. To this end, let $%
\mathbb{K}:\mathbb{R}\rightarrow \mathbb{R}_{+}$ be a fixed\textbf{\ }%
function, that will be called kernel, satisfying:\medskip 

\noindent $%
\begin{tabular}{ll}
$\left[ \mathbb{C}1\right] $ & $\mathbb{K}$ is non increasing and
right-continuous $\text{on }\mathbb{R};$ \\ 
$\left[ \mathbb{C}2\right] $ & $\mathbb{K}(s)=0$ for $s\notin \left[
0,1\right) $ and $\mathbb{K}(s)\geq 0$ for $s\in \left[ 0,1\right) ;$ \\ 
$\left[ \mathbb{C}3\right] $ & $\int_{\mathbb{R}}\mathbb{K}(s)ds=1;$ \\ 
$\left[ \mathbb{C}4\right] $ & $\mathbb{K}\text{ and its first and second
Lebesgue derivatives }\mathbb{K}^{\prime }\text{ and }\mathbb{K}^{\prime
\prime }\text{ are bounded on }\mathbb{R}.$%
\end{tabular}%
$\medskip

\noindent As examples of such functions \citep[see, e.g.,][]{GLW}, we have
the indicator kernel $\mathbb{K}=\mathbf{1}_{\left[ 0,1\right) }$ and the
biweight and triweight kernels respectively defined by%
\begin{equation}
\mathbb{K}_{2}(s):=\frac{15}{8}\left( 1-s^{2}\right) ^{2}\mathbf{1}_{\left\{
0\leq s<1\right\} },\text{ }\mathbb{K}_{3}(s):=\frac{35}{16}\left(
1-s^{2}\right) ^{3}\mathbf{1}_{\left\{ 0\leq s<1\right\} }.  \label{kernels}
\end{equation}%
For an overview of kernel estimation of the extreme value index with
complete data, one refers to, for instance, \cite{HLM} and \cite{CM-2010}.
By using Potter's inequalities, see e.g. Proposition B.1.10 in \cite{deHF06}%
, to the regularly varying function $\overline{\mathbf{F}}$\ together with
assumptions $\left[ \mathbb{C}1\right] $-$\left[ \mathbb{C}3\right] ,$ we
may readily show that%
\begin{equation}
\lim_{u\rightarrow \infty }\int_{u}^{\infty }x^{-1}\frac{\overline{\mathbf{F}%
}\left( x\right) }{\overline{\mathbf{F}}\left( u\right) }\mathbb{K}\left( 
\frac{\overline{\mathbf{F}}\left( x\right) }{\overline{\mathbf{F}}\left(
u\right) }\right) dx=\gamma _{1}\int_{0}^{\infty }\mathbb{K}(s)ds=\gamma
_{1}.  \label{limit-1}
\end{equation}%
An integration by parts yields%
\begin{equation}
\lim_{u\rightarrow \infty }\frac{1}{\overline{\mathbf{F}}\left( u\right) }%
\int_{u}^{\infty }g_{\mathbb{K}}\left( \frac{\overline{\mathbf{F}}\left(
x\right) }{\overline{\mathbf{F}}\left( u\right) }\right) \log \frac{x}{u}d%
\mathbf{F}\left( x\right) =\gamma _{1},  \label{limit-2}
\end{equation}%
where $g_{\mathbb{K}}$ denotes the Lebesgue derivative of the function $%
s\rightarrow \Psi _{\mathbb{K}}\left( s\right) :=s\mathbb{K}\left( s\right)
. $ Note that, for $\mathbb{K}=\mathbf{1}_{\left[ 0,1\right) },$ we have $g_{%
\mathbb{K}}=\mathbf{1}_{\left[ 0,1\right) },$ then the previous two limits
meet assertion $\left( 1.2.6\right) $ given in Theorem 1.2.2 by \cite{deHF06}%
. For kernels $\mathbb{K}_{2}$ and $\mathbb{K}_{3},$ we have%
\begin{equation}
g_{\mathbb{K}_{2}}(s):=\frac{15}{8}\left( 1-s^{2}\right) \left(
1-5s^{2}\right) \mathbf{1}_{\left\{ 0\leq s<1\right\} },\text{ }g_{\mathbb{K}%
_{3}}(s):=\frac{35}{16}\left( 1-s^{2}\right) ^{2}\left( 1-7s^{2}\right) 
\mathbf{1}_{\left\{ 0\leq s<1\right\} }.  \label{K}
\end{equation}%
Since $\overline{F}$ is regularly varying at infinity with tail index $%
\gamma >0,$ then $X_{n-k:n}$ tends to $\infty $ almost surely. By replacing,
in $(\ref{limit-2}),$ $u$ by $X_{n-k:n}$ and $\mathbf{F}$ by its empirical
counterpart $\mathbf{F}_{n},$ we get%
\begin{equation*}
\widehat{\gamma }_{1,\mathbb{K}}:=\frac{1}{\overline{\mathbf{F}}_{n}\left(
X_{n-k:n}\right) }\int_{X_{n-k:n}}^{\infty }g_{\mathbb{K}}\left( \frac{%
\overline{\mathbf{F}}_{n}\left( x\right) }{\overline{\mathbf{F}}_{n}\left(
X_{n-k:n}\right) }\right) \log \frac{x}{X_{n-k:n}}d\mathbf{F}_{n}\left(
x\right) ,
\end{equation*}%
as a kernel estimator for $\gamma _{1}.$ Next, we give an explicit formula
for $\widehat{\gamma }_{1,\mathbb{K}}.$ Since $\overline{\mathbf{F}}$ and $%
\overline{\mathbf{G}}$ are regularly varying at infinity with tail indices $%
\gamma _{1}>0$ and $\gamma _{2}>0$ respectively, then their right endpoints
are infinite and so they are equal. Hence, from \cite{W-85}, we may write%
\begin{equation}
\int_{x}^{\infty }\frac{d\mathbf{F}\left( y\right) }{\mathbf{F}\left(
y\right) }=\int_{x}^{\infty }\frac{dF\left( y\right) }{C\left( y\right) },
\label{int}
\end{equation}%
where $C\left( z\right) :=\mathbf{P}\left( X\leq z\leq Y\right) $ is the
theoretical counterpart of $C_{n}\left( z\right) $ defined in $(\ref{Cn}).$
Differentiating $\left( \ref{int}\right) $ leads to the following crucial
equation $C\left( x\right) d\mathbf{F}\left( x\right) =\mathbf{F}\left(
x\right) dF\left( x\right) ,$ which implies that%
\begin{equation}
C_{n}\left( x\right) d\mathbf{F}_{n}\left( x\right) =\mathbf{F}_{n}\left(
x\right) dF_{n}\left( x\right) ,  \label{relation}
\end{equation}%
with $F_{n}\left( x\right) :=n^{-1}\sum\limits_{i=1}^{n}\mathbf{1}\left(
X_{i}\leq x\right) $ being the empirical counterpart of $F\left( x\right) .$
This allow us to rewrite $\widehat{\gamma }_{1,\mathbb{K}}$ into 
\begin{equation*}
\widehat{\gamma }_{1,\mathbb{K}}=\frac{1}{\overline{\mathbf{F}}_{n}\left(
X_{n-k:n}\right) }\int_{X_{n-k:n}}^{\infty }\frac{\mathbf{F}_{n}\left(
x\right) }{C_{n}\left( x\right) }g_{\mathbb{K}}\left( \frac{\overline{%
\mathbf{F}}_{n}\left( x\right) }{\overline{\mathbf{F}}_{n}\left(
X_{n-k:n}\right) }\right) \log \frac{x}{X_{n-k:n}}dF_{n}\left( x\right) ,
\end{equation*}%
which is equal to%
\begin{equation*}
\frac{1}{n\overline{\mathbf{F}}_{n}\left( X_{n-k:n}\right) }%
\sum\limits_{i=1}^{k}\dfrac{\mathbf{F}_{n}\left( X_{n-i+1:n}\right) }{%
C_{n}\left( X_{n-i+1:n}\right) }g_{\mathbb{K}}\left( \dfrac{\overline{%
\mathbf{F}}_{n}\left( X_{n-i+1:n}\right) }{\overline{\mathbf{F}}_{n}\left(
X_{n-k:n}\right) }\right) \log \dfrac{X_{n-i+1:n}}{X_{n-k:n}}.
\end{equation*}%
In view of equation $(\ref{relation}),$ \cite{BchMN-15b} showed that%
\begin{equation*}
\overline{\mathbf{F}}_{n}\left( X_{n-k:n}\right) =\frac{1}{n}%
\sum\limits_{i=1}^{k}\dfrac{\mathbf{F}_{n}\left( X_{n-i+1:n}\right) }{%
C_{n}\left( X_{n-i+1:n}\right) }.
\end{equation*}%
Thereby, by setting $a_{n}^{\left( i\right) }:=\mathbf{F}_{n}\left(
X_{n-i+1:n}\right) /C_{n}\left( X_{n-i+1:n}\right) ,$ we end up with the
final formula of our new kernel estimator 
\begin{equation}
\widehat{\gamma }_{1,\mathbb{K}}=\frac{\sum\limits_{i=1}^{k}a_{n}^{\left(
i\right) }g_{\mathbb{K}}\left( \dfrac{\overline{\mathbf{F}}_{n}\left(
X_{n-i+1:n}\right) }{\overline{\mathbf{F}}_{n}\left( X_{n-k:n}\right) }%
\right) \log \dfrac{X_{n-i+1:n}}{X_{n-k:n}}}{\sum\limits_{i=1}^{k}a_{n}^{%
\left( i\right) }}.  \label{kernel-estimator}
\end{equation}%
Note that in the complete data situation, $\mathbf{F}_{n}$ is equal to $%
C_{n} $ and both reduce to the classical empirical df. As a result, we have
in that case $a_{n}^{\left( i\right) }=1$ and $\overline{\mathbf{F}}%
_{n}\left( X_{n-i,n}\right) /\overline{\mathbf{F}}_{n}\left(
X_{n-k,n}\right) =i/k$ meaning that $\widehat{\gamma }_{1,\mathbb{K}%
}=k^{-1}\dsum\nolimits_{i=1}^{k}g_{\mathbb{K}}\left( \dfrac{i-1}{k}\right)
\log \left( X_{n-i+1:n}/X_{n-k:n}\right) .$ By applying the mean value
theorem to function $\Psi _{\mathbb{K}},$ we get 
\begin{equation*}
\dfrac{i}{k}\mathbb{K}\left( \dfrac{i}{k}\right) -\dfrac{i-1}{k}\mathbb{K}%
\left( \dfrac{i-1}{k}\right) =\dfrac{1}{k}g_{\mathbb{K}}\left( \dfrac{i-1}{k}%
\right) +O\left( \dfrac{1}{k^{2}}\right) ,\text{ as }N\rightarrow \infty .
\end{equation*}%
It follows that%
\begin{equation*}
\widehat{\gamma }_{1,\mathbb{K}}=\sum\limits_{i=1}^{k}\left\{ \dfrac{i}{k}%
\mathbb{K}\left( \dfrac{i}{k}\right) -\dfrac{i-1}{k}\mathbb{K}\left( \dfrac{%
i-1}{k}\right) \right\} \log \dfrac{X_{n-i+1:n}}{X_{n-k:n}}+O\left( \dfrac{1%
}{k}\right) \widehat{\gamma }_{1}^{Hill},
\end{equation*}%
where $\widehat{\gamma }_{1}^{Hill}:=k^{-1}\sum\limits_{i=1}^{k}\log \left(
X_{n-i+1:n}/X_{n-k:n}\right) $ is Hill's estimator of the tail index $\gamma
_{1}.$ In view of the consistency of $\widehat{\gamma }_{1}^{Hill}$ %
\citep{Mas82}, we obtain%
\begin{equation*}
\widehat{\gamma }_{1,\mathbb{K}}=\sum\limits_{i=1}^{k}\dfrac{i}{k}\mathbb{K}%
\left( \dfrac{i}{k}\right) \log \dfrac{X_{n-i+1:n}}{X_{n-i:n}}+O_{\mathbf{p}%
}\left( \dfrac{1}{k}\right) ,\text{ as }N\rightarrow \infty ,
\end{equation*}%
which is an approximation of the above talked about CDM's kernel estimator
of the tail index $\gamma _{1}$ with untrucated data. The rest of the paper
is organized as follows. In Section \ref{sec2}, we provide our main result,
namely the asymptotic normality of $\widehat{\gamma }_{1,\mathbb{K}},$ whose
proof is postponed to Section \ref{sec4}. The finite sample behavior of the
proposed estimator is checked by simulation in Section \ref{sec3}, where a
comparison with the aforementioned already existing ones is made as well.
Finally a lemma that is instrumental to the proof is given in the Appendix.

\section{\textbf{Main results\label{sec2}}}

\noindent \noindent It is very well known that, in the context of extreme
value analysis, weak approximations are achieved in the second-order
framework \citep[see, e.g.,
][]{deHS96}. Thus, it seems quite natural to suppose that $\mathbf{F}$\ and $%
\mathbf{G}$\ satisfy the second-order condition of regular variation, which
we express in terms of the tail quantile functions pertaining to both df's.
That is, we assume that for $x>0,$ we have%
\begin{equation}
\underset{t\rightarrow \infty }{\lim }\dfrac{\mathbb{U}_{\mathbf{F}}\left(
tx\right) /\mathbb{U}_{\mathbf{F}}\left( t\right) -x^{\gamma _{1}}}{\mathbf{A%
}_{\mathbf{F}}\left( t\right) }=x^{\gamma _{1}}\dfrac{x^{\tau _{1}}-1}{\tau
_{1}},  \label{second-order}
\end{equation}%
and%
\begin{equation}
\underset{t\rightarrow \infty }{\lim }\dfrac{\mathbb{U}_{\mathbf{G}}\left(
tx\right) /\mathbb{U}_{\mathbf{G}}\left( t\right) -x^{\gamma _{2}}}{\mathbf{A%
}_{\mathbf{G}}\left( t\right) }=x^{\gamma _{2}}\dfrac{x^{\tau _{2}}-1}{\tau
_{2}},  \label{second-orderG}
\end{equation}%
where $\tau _{1},\tau _{2}<0$\ are the second-order parameters and $\mathbf{A%
}_{\mathbf{F}},$ $\mathbf{A}_{\mathbf{G}}$\ are functions tending to zero
and not changing signs near infinity with regularly varying absolute values
at infinity with indices $\tau _{1},$ $\tau _{2}$ respectively. For any df $%
H,$ the function $\mathbb{U}_{H}\left( t\right) :=H^{\leftarrow }\left(
1-1/t\right) ,$ $t>1,$ stands for the tail quantile function, with $%
H^{\leftarrow }\left( u\right) :=\inf \left\{ v:H\left( v\right) \geq
u\right\} ,$ $0<u<1,$ denoting the quantile function. For convenience, we
set $\mathbf{A}_{\mathbf{F}}^{\ast }\left( t\right) :=\mathbf{A}_{\mathbf{F}%
}\left( 1/\overline{\mathbf{F}}\left( \mathbb{U}_{F}\left( t\right) \right)
\right) .$

\begin{theorem}
\label{Theorem1}Assume that the second-order conditions of regular variation 
$(\ref{second-order})$\ and $(\ref{second-orderG})$\ hold with $\gamma
_{1}<\gamma _{2},$ and let $\mathbb{K}$ be a kernel function satisfying
assumptions $\left[ \mathbb{C}1\right] $-$\left[ \mathbb{C}4\right] $ and $%
k_{N}$ an integer sequence such that $k_{N}\rightarrow \infty $ and $%
k_{N}/N\rightarrow 0,$ as $N\rightarrow \infty .$ Then, there exist a
function $\mathbf{A}_{0}\left( t\right) \sim \mathbf{A}_{\mathbf{F}}^{\ast
}\left( t\right) ,$ as $t\rightarrow \infty ,$ and a standard Wiener process 
$\left\{ \mathbf{W}\left( s\right) ;\text{ }s\geq 0\right\} ,$\ defined on
the probability space $\left( \Omega ,\mathcal{A},\mathbf{P}\right) $ such
that%
\begin{eqnarray*}
&&\sqrt{k}\left( \widehat{\gamma }_{1,\mathbb{K}}-\gamma _{1}\right)  \\
&=&\left( \gamma ^{2}/\gamma _{1}\right) \int_{0}^{1}s^{-1}\mathbf{W}\left(
s\right) d\left\{ s\varphi _{\mathbb{K}}\left( s\right) \right\} +\sqrt{k}%
\mathbf{A}_{0}\left( n/k\right) \int_{0}^{1}s^{-\tau _{1}}\mathbb{K}\left(
s\right) ds+o_{\mathbf{P}}\left( 1\right) ,
\end{eqnarray*}%
provided that $\sqrt{k_{N}}\mathbf{A}_{\mathbf{0}}\left( N/k_{N}\right)
=O\left( 1\right) ,$ as $N\rightarrow \infty ,$ where 
\begin{equation*}
\varphi _{\mathbb{K}}\left( s\right) :=s^{-1}\int_{0}^{s}t^{-\gamma /\gamma
_{2}}\left\{ \mathbb{K}\left( t^{\gamma /\gamma _{1}}\right) -\frac{\gamma
_{1}}{\gamma _{2}}t^{-\gamma _{2}/\gamma _{1}}\mathbb{K}\left( t^{\gamma
/\gamma _{1}}\right) +t^{\gamma /\gamma _{1}}\mathbb{K}^{\prime }\left(
t^{\gamma /\gamma _{1}}\right) \right\} dt.
\end{equation*}%
If in addition we suppose that $\sqrt{k_{N}}\mathbf{A}_{\mathbf{0}}\left(
N/k_{N}\right) \rightarrow \lambda ,$ then $\sqrt{k}\left( \widehat{\gamma }%
_{1,\mathbb{K}}-\gamma _{1}\right) \overset{\mathcal{D}}{\rightarrow }%
\mathcal{N}\left( \mu _{\mathbb{K}},\sigma _{\mathbb{K}}^{2}\right) ,$ as $%
N\rightarrow \infty ,$ where $\mu _{\mathbb{K}}:=\lambda
\int_{0}^{1}s^{-\tau _{1}}\mathbb{K}\left( s\right) ds$ and $\sigma _{%
\mathbb{K}}^{2}:=\left( \gamma ^{2}/\gamma _{1}\right)
^{2}\int_{0}^{1}\varphi _{\mathbb{K}}^{2}\left( s\right) ds.$
\end{theorem}

\begin{remark}
A very large value of $\gamma _{2}$\textbf{\ }yields a $\gamma $-value that
is very close to $\gamma _{1},$ meaning that the really observed sample is
almost the whole dataset. In other words, the complete data case corresponds
to the situation when $1/\gamma _{2}\equiv 0,$\ in which case we have $%
\gamma \equiv \gamma _{1}.$\ It follows that in that case $\varphi _{\mathbb{%
K}}\left( s\right) =\gamma _{1}s^{-1}\int_{0}^{s}\left\{ \mathbb{K}\left(
t\right) +t\mathbb{K}^{\prime }\left( t\right) \right\} dt=\gamma
_{1}s^{-1}\int_{0}^{s}d\left\{ t\mathbb{K}\left( t\right) \right\} =\gamma
_{1}\mathbb{K}\left( s\right) ,$ and therefore $\sigma _{\mathbb{K}%
}^{2}=\gamma _{1}^{2}\int_{0}^{1}\mathbb{K}^{2}\left( s\right) ds,$ which
agrees with the asymptotic variance given in Theorem 1 of \cite{CDM}.
\end{remark}

\section{\textbf{Simulation study\label{sec3}}}

\noindent In this section, we check the finite sample behavior of $\widehat{%
\gamma }_{1,\mathbb{K}}$ and, at the same time, we compare it with $\widehat{%
\gamma }_{1}$ and $\widehat{\gamma }_{1}^{\left( \mathbf{GS}\right) }$
respectively proposed by \cite{BchMN-15b} and \cite{GS2015} and defined in $(%
\ref{BMN})$ and $(\ref{GS}).$ To this end, we consider two sets of truncated
and truncation data, both drawn from Burr's model:%
\begin{equation*}
\overline{\mathbf{F}}\left( x\right) =\left( 1+x^{1/\delta }\right)
^{-\delta /\gamma _{1}},\text{ }\overline{\mathbf{G}}\left( x\right) =\left(
1+x^{1/\delta }\right) ^{-\delta /\gamma _{2}},\text{ }x\geq 0,
\end{equation*}%
where $\delta ,\gamma _{1},\gamma _{2}>0.$ The corresponding percentage of
observed data is equal to $p=\gamma _{2}/(\gamma _{1}+\gamma _{2}).$ We fix $%
\delta =1/4$ and choose the values $0.6$ and $0.8$ for $\gamma _{1}$ and $%
70\%,$ $80\%$ and $90\%$ for $p.$ For each couple $\left( \gamma
_{1},p\right) ,$ we solve the equation $p=\gamma _{2}/(\gamma _{1}+\gamma
_{2})$ to get the pertaining $\gamma _{2}$-value. For the construction of
our estimator $\widehat{\gamma }_{1,\mathbb{K}},$ we select the biweight and
the triweight kernel functions defined in $(\ref{kernels}).$ We vary the
common size $N$ of both samples $\left( \mathbf{X}_{1},...,\mathbf{X}%
_{N}\right) $ and $\left( \mathbf{Y}_{1},...,\mathbf{Y}_{N}\right) ,$ then
for each size, we generate $1000$ independent replicates. Our overall
results are taken as the empirical means of the results obtained through all
repetitions. To determine the optimal number of top statistics used in the
computation of each one of the three estimators, we use the algorithm of 
\cite{ReTo7}, page 137. Our illustration and comparison are made with
respect to the estimators absolute biases (abs bias) and the roots of their
mean squared errors (rmse). We summarize the simulation results in Tables %
\ref{Tab1} and \ref{Tab2} for $\gamma _{1}=0.6$ and in Tables \ref{Tab3} and %
\ref{Tab4} for $\gamma _{1}=0.8.$ In light of all four tables, we first note
that, as expected, the estimation accuracy of all estimators decreases when
the truncation percentage increases. Second, with regard to the bias, the
comparison definitely is in favour of the newly proposed tail index
estimator $\widehat{\gamma }_{1,\mathbb{K}},$ whereas it is not as clear-cut
when the rmse is considered. Indeed, the kernel estimator preforms better
than the other pair as far as small samples are concerned while for large
datasets, it is $\widehat{\gamma }_{1}^{\left( \mathbf{GS}\right) }$ that
seems to have the least rmse but with greater bias. As an overall
conclusion, one may say that, for case studies where not so many data are at
one's disposal, the kernel estimator $\widehat{\gamma }_{1,\mathbb{K}}$ is
the most suitable among the three estimators.

\bigskip\ 
\begin{table}[tbp] \centering%
\begin{tabular}{cccccccc}
\hline
\multicolumn{8}{c}{${\small p=0.7}$} \\ \hline
&  & \multicolumn{2}{||c}{$\widehat{\gamma }_{1,\mathbb{K}}$} & 
\multicolumn{2}{|c}{$\widehat{\gamma }_{1}$} & \multicolumn{2}{|c}{$\widehat{%
\gamma }_{1}^{GS}$} \\ \hline
$N$ & $n$ & \multicolumn{1}{||c}{\small abs bias} & {\small rmse} & 
\multicolumn{1}{|c}{\small abs bias} & {\small rmse} & \multicolumn{1}{|c}%
{\small abs bias} & {\small rmse} \\ \hline\hline
\multicolumn{1}{r}{$150$} & \multicolumn{1}{r}{$104$} & \multicolumn{1}{||c}{%
$0.073$} & $0.665$ & \multicolumn{1}{|c}{$0.133$} & $0.408$ & 
\multicolumn{1}{|c}{$0.136$} & $3.341$ \\ 
\multicolumn{1}{r}{$200$} & \multicolumn{1}{r}{$140$} & \multicolumn{1}{||c}{%
$0.008$} & $0.614$ & \multicolumn{1}{|c}{$0.152$} & $0.392$ & 
\multicolumn{1}{|c}{$0.258$} & $1.647$ \\ 
\multicolumn{1}{r}{$300$} & \multicolumn{1}{r}{$210$} & \multicolumn{1}{||c}{%
$0.003$} & $0.467$ & \multicolumn{1}{|c}{$0.095$} & $0.321$ & 
\multicolumn{1}{|c}{$0.102$} & $0.962$ \\ 
\multicolumn{1}{r}{$500$} & \multicolumn{1}{r}{$349$} & \multicolumn{1}{||c}{%
$0.007$} & $0.439$ & \multicolumn{1}{|c}{$0.063$} & $0.296$ & 
\multicolumn{1}{|c}{$0.022$} & $0.409$ \\ 
\multicolumn{1}{r}{$1000$} & \multicolumn{1}{r}{$699$} & 
\multicolumn{1}{||c}{$0.020$} & $0.284$ & \multicolumn{1}{|c}{$0.042$} & $%
0.210$ & \multicolumn{1}{|c}{$0.023$} & $0.211$ \\ 
\multicolumn{1}{r}{$1500$} & \multicolumn{1}{r}{$1049$} & 
\multicolumn{1}{||c}{$0.009$} & $0.255$ & \multicolumn{1}{|c}{$0.024$} & $%
0.189$ & \multicolumn{1}{|c}{$0.013$} & $0.142$ \\ 
\multicolumn{1}{r}{$2000$} & \multicolumn{1}{r}{$1399$} & 
\multicolumn{1}{||c}{$0.011$} & $0.245$ & \multicolumn{1}{|c}{$0.018$} & $%
0.177$ & \multicolumn{1}{|c}{$0.013$} & $0.116$ \\ \hline\hline
\multicolumn{8}{c}{${\small p=0.8}$} \\ \hline
\multicolumn{1}{r}{$150$} & \multicolumn{1}{r}{$120$} & \multicolumn{1}{||c}{%
$0.054$} & $0.608$ & \multicolumn{1}{|c}{$0.093$} & $0.398$ & 
\multicolumn{1}{|c}{$0.100$} & $0.989$ \\ 
\multicolumn{1}{r}{$200$} & \multicolumn{1}{r}{$160$} & \multicolumn{1}{||c}{%
$0.030$} & $0.520$ & \multicolumn{1}{|c}{$0.085$} & $0.353$ & 
\multicolumn{1}{|c}{$0.109$} & $0.488$ \\ 
\multicolumn{1}{r}{$300$} & \multicolumn{1}{r}{$239$} & \multicolumn{1}{||c}{%
$0.022$} & $0.467$ & \multicolumn{1}{|c}{$0.067$} & $0.322$ & 
\multicolumn{1}{|c}{$0.069$} & $0.353$ \\ 
\multicolumn{1}{r}{$500$} & \multicolumn{1}{r}{$399$} & \multicolumn{1}{||c}{%
$0.002$} & $0.340$ & \multicolumn{1}{|c}{$0.049$} & $0.240$ & 
\multicolumn{1}{|c}{$0.040$} & $0.196$ \\ 
\multicolumn{1}{r}{$1000$} & \multicolumn{1}{r}{$799$} & 
\multicolumn{1}{||c}{$0.013$} & $0.217$ & \multicolumn{1}{|c}{$0.033$} & $%
0.168$ & \multicolumn{1}{|c}{$0.029$} & $0.135$ \\ 
\multicolumn{1}{r}{$1500$} & \multicolumn{1}{r}{$1199$} & 
\multicolumn{1}{||c}{$0.003$} & $0.190$ & \multicolumn{1}{|c}{$0.017$} & $%
0.140$ & \multicolumn{1}{|c}{$0.019$} & $0.109$ \\ 
\multicolumn{1}{r}{$2000$} & \multicolumn{1}{r}{$1599$} & 
\multicolumn{1}{||c}{$0.005$} & $0.149$ & \multicolumn{1}{|c}{$0.011$} & $%
0.113$ & \multicolumn{1}{|c}{$0.005$} & $0.095$ \\ \hline\hline
\multicolumn{8}{c}{${\small p=0.9}$} \\ \hline
\multicolumn{1}{r}{$150$} & \multicolumn{1}{r}{$134$} & \multicolumn{1}{||c}{%
$0.031$} & $0.492$ & \multicolumn{1}{|c}{$0.082$} & $0.387$ & 
\multicolumn{1}{|c}{$0.149$} & $2.740$ \\ 
\multicolumn{1}{r}{$200$} & \multicolumn{1}{r}{$180$} & \multicolumn{1}{||c}{%
$0.019$} & $0.404$ & \multicolumn{1}{|c}{$0.069$} & $0.313$ & 
\multicolumn{1}{|c}{$0.072$} & $0.334$ \\ 
\multicolumn{1}{r}{$300$} & \multicolumn{1}{r}{$270$} & \multicolumn{1}{||c}{%
$0.016$} & $0.299$ & \multicolumn{1}{|c}{$0.051$} & $0.238$ & 
\multicolumn{1}{|c}{$0.043$} & $0.231$ \\ 
\multicolumn{1}{r}{$500$} & \multicolumn{1}{r}{$449$} & \multicolumn{1}{||c}{%
$0.002$} & $0.236$ & \multicolumn{1}{|c}{$0.045$} & $0.176$ & 
\multicolumn{1}{|c}{$0.037$} & $0.160$ \\ 
\multicolumn{1}{r}{$1000$} & \multicolumn{1}{r}{$899$} & 
\multicolumn{1}{||c}{$0.006$} & $0.163$ & \multicolumn{1}{|c}{$0.024$} & $%
0.131$ & \multicolumn{1}{|c}{$0.020$} & $0.123$ \\ 
\multicolumn{1}{r}{$1500$} & \multicolumn{1}{r}{$1350$} & 
\multicolumn{1}{||c}{$0.010$} & $0.131$ & \multicolumn{1}{|c}{$0.021$} & $%
0.103$ & \multicolumn{1}{|c}{$0.018$} & $0.093$ \\ 
\multicolumn{1}{r}{$2000$} & \multicolumn{1}{r}{$1799$} & 
\multicolumn{1}{||c}{$0.002$} & $0.116$ & \multicolumn{1}{|c}{$0.010$} & $%
0.088$ & \multicolumn{1}{|c}{$0.009$} & $0.078$ \\ \hline
&  &  &  &  &  &  & 
\end{tabular}%
\caption{Biweight-kernel estimation results for the shape parameter
$\gamma_{1}=0.6$ of Burr's model based on 1000 right-truncated samples,
along with other existing estimators}\label{Tab1}%
\end{table}%

\begin{table}[tbp] \centering%
\begin{tabular}{cccccccc}
\hline
\multicolumn{8}{c}{${\small p=0.7}$} \\ \hline
&  & \multicolumn{2}{||c}{$\widehat{\gamma }_{1,\mathbb{K}}$} & 
\multicolumn{2}{|c}{$\widehat{\gamma }_{1}$} & \multicolumn{2}{|c}{$\widehat{%
\gamma }_{1}^{GS}$} \\ \hline
$N$ & $n$ & \multicolumn{1}{||c}{\small abs bias} & {\small rmse} & 
\multicolumn{1}{|c}{\small abs bias} & {\small rmse} & \multicolumn{1}{|c}%
{\small abs bias} & {\small rmse} \\ \hline\hline
\multicolumn{1}{r}{$150$} & \multicolumn{1}{r}{$104$} & \multicolumn{1}{||c}{%
$0.134$} & $0.808$ & \multicolumn{1}{|c}{$0.142$} & $0.408$ & 
\multicolumn{1}{|c}{$0.245$} & $1.242$ \\ 
\multicolumn{1}{r}{$200$} & \multicolumn{1}{r}{$139$} & \multicolumn{1}{||c}{%
$0.097$} & $0.705$ & \multicolumn{1}{|c}{$0.129$} & $0.373$ & 
\multicolumn{1}{|c}{$0.184$} & $0.857$ \\ 
\multicolumn{1}{r}{$300$} & \multicolumn{1}{r}{$209$} & \multicolumn{1}{||c}{%
$0.045$} & $0.566$ & \multicolumn{1}{|c}{$0.090$} & $0.313$ & 
\multicolumn{1}{|c}{$0.091$} & $0.582$ \\ 
\multicolumn{1}{r}{$500$} & \multicolumn{1}{r}{$349$} & \multicolumn{1}{||c}{%
$0.002$} & $0.430$ & \multicolumn{1}{|c}{$0.074$} & $0.268$ & 
\multicolumn{1}{|c}{$0.064$} & $0.550$ \\ 
\multicolumn{1}{r}{$1000$} & \multicolumn{1}{r}{$699$} & 
\multicolumn{1}{||c}{$0.003$} & $0.399$ & \multicolumn{1}{|c}{$0.031$} & $%
0.237$ & \multicolumn{1}{|c}{$0.023$} & $0.161$ \\ 
\multicolumn{1}{r}{$1500$} & \multicolumn{1}{r}{$1050$} & 
\multicolumn{1}{||c}{$0.010$} & $0.362$ & \multicolumn{1}{|c}{$0.013$} & $%
0.217$ & \multicolumn{1}{|c}{$0.010$} & $0.130$ \\ 
\multicolumn{1}{r}{$2000$} & \multicolumn{1}{r}{$1401$} & 
\multicolumn{1}{||c}{$0.010$} & $0.244$ & \multicolumn{1}{|c}{$0.018$} & $%
0.164$ & \multicolumn{1}{|c}{$0.009$} & $0.117$ \\ \hline\hline
\multicolumn{8}{c}{${\small p=0.8}$} \\ \hline
\multicolumn{1}{r}{$150$} & \multicolumn{1}{r}{$119$} & \multicolumn{1}{||c}{%
$0.096$} & $0.730$ & \multicolumn{1}{|c}{$0.109$} & $0.397$ & 
\multicolumn{1}{|c}{$0.117$} & $0.729$ \\ 
\multicolumn{1}{r}{$200$} & \multicolumn{1}{r}{$159$} & \multicolumn{1}{||c}{%
$0.060$} & $0.580$ & \multicolumn{1}{|c}{$0.091$} & $0.340$ & 
\multicolumn{1}{|c}{$0.108$} & $0.874$ \\ 
\multicolumn{1}{r}{$300$} & \multicolumn{1}{r}{$239$} & \multicolumn{1}{||c}{%
$0.037$} & $0.496$ & \multicolumn{1}{|c}{$0.067$} & $0.315$ & 
\multicolumn{1}{|c}{$0.080$} & $0.490$ \\ 
\multicolumn{1}{r}{$500$} & \multicolumn{1}{r}{$399$} & \multicolumn{1}{||c}{%
$0.009$} & $0.303$ & \multicolumn{1}{|c}{$0.057$} & $0.231$ & 
\multicolumn{1}{|c}{$0.047$} & $0.280$ \\ 
\multicolumn{1}{r}{$1000$} & \multicolumn{1}{r}{$799$} & 
\multicolumn{1}{||c}{$0.001$} & $0.265$ & \multicolumn{1}{|c}{$0.027$} & $%
0.177$ & \multicolumn{1}{|c}{$0.021$} & $0.139$ \\ 
\multicolumn{1}{r}{$1500$} & \multicolumn{1}{r}{$1199$} & 
\multicolumn{1}{||c}{$0.008$} & $0.194$ & \multicolumn{1}{|c}{$0.018$} & $%
0.139$ & \multicolumn{1}{|c}{$0.015$} & $0.109$ \\ 
\multicolumn{1}{r}{$2000$} & \multicolumn{1}{r}{$1600$} & 
\multicolumn{1}{||c}{$0.001$} & $0.183$ & \multicolumn{1}{|c}{$0.013$} & $%
0.124$ & \multicolumn{1}{|c}{$0.012$} & $0.095$ \\ \hline\hline
\multicolumn{8}{c}{${\small p=0.9}$} \\ \hline
\multicolumn{1}{r}{$150$} & \multicolumn{1}{r}{$134$} & \multicolumn{1}{||c}{%
$0.066$} & $0.660$ & \multicolumn{1}{|c}{$0.080$} & $0.392$ & 
\multicolumn{1}{|c}{$0.081$} & $0.450$ \\ 
\multicolumn{1}{r}{$200$} & \multicolumn{1}{r}{$179$} & \multicolumn{1}{||c}{%
$0.047$} & $0.454$ & \multicolumn{1}{|c}{$0.061$} & $0.314$ & 
\multicolumn{1}{|c}{$0.061$} & $0.359$ \\ 
\multicolumn{1}{r}{$300$} & \multicolumn{1}{r}{$270$} & \multicolumn{1}{||c}{%
$0.003$} & $0.299$ & \multicolumn{1}{|c}{$0.064$} & $0.243$ & 
\multicolumn{1}{|c}{$0.062$} & $0.230$ \\ 
\multicolumn{1}{r}{$500$} & \multicolumn{1}{r}{$449$} & \multicolumn{1}{||c}{%
$0.001$} & $0.226$ & \multicolumn{1}{|c}{$0.043$} & $0.174$ & 
\multicolumn{1}{|c}{$0.037$} & $0.164$ \\ 
\multicolumn{1}{r}{$1000$} & \multicolumn{1}{r}{$899$} & 
\multicolumn{1}{||c}{$0.009$} & $0.175$ & \multicolumn{1}{|c}{$0.016$} & $%
0.124$ & \multicolumn{1}{|c}{$0.014$} & $0.113$ \\ 
\multicolumn{1}{r}{$1500$} & \multicolumn{1}{r}{$1350$} & 
\multicolumn{1}{||c}{$0.002$} & $0.146$ & \multicolumn{1}{|c}{$0.017$} & $%
0.108$ & \multicolumn{1}{|c}{$0.017$} & $0.098$ \\ 
\multicolumn{1}{r}{$2000$} & \multicolumn{1}{r}{$1799$} & 
\multicolumn{1}{||c}{$0.003$} & $0.134$ & \multicolumn{1}{|c}{$0.010$} & $%
0.093$ & \multicolumn{1}{|c}{$0.008$} & $0.081$ \\ \hline
&  &  &  &  &  &  & 
\end{tabular}%
\caption{Triweight-kernel estimation results for the shape parameter
$\gamma_{1}=0.6$ of Burr's model based on 1000 right-truncated samples,
along with other existing estimators}\label{Tab2}%
\end{table}%

\begin{table}[tbp] \centering%
\begin{tabular}{cccccccc}
\hline
\multicolumn{8}{c}{${\small p=0.7}$} \\ \hline
&  & \multicolumn{2}{||c}{$\widehat{\gamma }_{1,\mathbb{K}}$} & 
\multicolumn{2}{|c}{$\widehat{\gamma }_{1}$} & \multicolumn{2}{|c}{$\widehat{%
\gamma }_{1}^{GS}$} \\ \hline
$N$ & $n$ & \multicolumn{1}{||c}{\small abs bias} & {\small rmse} & 
\multicolumn{1}{|c}{\small abs bias} & {\small rmse} & \multicolumn{1}{|c}%
{\small abs bias} & {\small rmse} \\ \hline\hline
\multicolumn{1}{r}{$150$} & \multicolumn{1}{r}{$105$} & \multicolumn{1}{||r}{%
$0.090$} & \multicolumn{1}{r}{$0.893$} & \multicolumn{1}{|r}{$0.187$} & 
\multicolumn{1}{r}{$0.548$} & \multicolumn{1}{|r}{$0.294$} & 
\multicolumn{1}{r}{$2.126$} \\ 
\multicolumn{1}{r}{$200$} & \multicolumn{1}{r}{$139$} & \multicolumn{1}{||r}{%
$0.014$} & \multicolumn{1}{r}{$0.863$} & \multicolumn{1}{|r}{$0.199$} & 
\multicolumn{1}{r}{$0.542$} & \multicolumn{1}{|r}{$0.316$} & 
\multicolumn{1}{r}{$1.351$} \\ 
\multicolumn{1}{r}{$300$} & \multicolumn{1}{r}{$210$} & \multicolumn{1}{||r}{%
$0.022$} & \multicolumn{1}{r}{$0.573$} & \multicolumn{1}{|r}{$0.140$} & 
\multicolumn{1}{r}{$0.412$} & \multicolumn{1}{|r}{$0.173$} & 
\multicolumn{1}{r}{$0.812$} \\ 
\multicolumn{1}{r}{$500$} & \multicolumn{1}{r}{$349$} & \multicolumn{1}{||r}{%
$0.031$} & \multicolumn{1}{r}{$0.519$} & \multicolumn{1}{|r}{$0.103$} & 
\multicolumn{1}{r}{$0.372$} & \multicolumn{1}{|r}{$0.053$} & 
\multicolumn{1}{r}{$0.593$} \\ 
\multicolumn{1}{r}{$1000$} & \multicolumn{1}{r}{$699$} & 
\multicolumn{1}{||r}{$0.004$} & \multicolumn{1}{r}{$0.462$} & 
\multicolumn{1}{|r}{$0.042$} & \multicolumn{1}{r}{$0.324$} & 
\multicolumn{1}{|r}{$0.020$} & \multicolumn{1}{r}{$0.253$} \\ 
\multicolumn{1}{r}{$1500$} & \multicolumn{1}{r}{$1049$} & 
\multicolumn{1}{||r}{$0.017$} & \multicolumn{1}{r}{$0.356$} & 
\multicolumn{1}{|r}{$0.031$} & \multicolumn{1}{r}{$0.255$} & 
\multicolumn{1}{|r}{$0.020$} & \multicolumn{1}{r}{$0.174$} \\ 
\multicolumn{1}{r}{$2000$} & \multicolumn{1}{r}{$1399$} & 
\multicolumn{1}{||r}{$0.008$} & \multicolumn{1}{r}{$0.424$} & 
\multicolumn{1}{|r}{$0.017$} & \multicolumn{1}{r}{$0.267$} & 
\multicolumn{1}{|r}{$0.017$} & \multicolumn{1}{r}{$0.150$} \\ \hline\hline
\multicolumn{8}{c}{${\small p=0.8}$} \\ \hline
\multicolumn{1}{r}{$150$} & \multicolumn{1}{r}{$120$} & \multicolumn{1}{||r}{%
$0.088$} & \multicolumn{1}{r}{$0.862$} & \multicolumn{1}{|r}{$0.122$} & 
\multicolumn{1}{r}{$0.553$} & \multicolumn{1}{|r}{$0.248$} & 
\multicolumn{1}{r}{$1.947$} \\ 
\multicolumn{1}{r}{$200$} & \multicolumn{1}{r}{$159$} & \multicolumn{1}{||r}{%
$0.040$} & \multicolumn{1}{r}{$0.684$} & \multicolumn{1}{|r}{$0.121$} & 
\multicolumn{1}{r}{$0.472$} & \multicolumn{1}{|r}{$0.178$} & 
\multicolumn{1}{r}{$1.143$} \\ 
\multicolumn{1}{r}{$300$} & \multicolumn{1}{r}{$239$} & \multicolumn{1}{||r}{%
$0.006$} & \multicolumn{1}{r}{$0.516$} & \multicolumn{1}{|r}{$0.084$} & 
\multicolumn{1}{r}{$0.406$} & \multicolumn{1}{|r}{$0.099$} & 
\multicolumn{1}{r}{$0.494$} \\ 
\multicolumn{1}{r}{$500$} & \multicolumn{1}{r}{$399$} & \multicolumn{1}{||r}{%
$0.022$} & \multicolumn{1}{r}{$0.372$} & \multicolumn{1}{|r}{$0.078$} & 
\multicolumn{1}{r}{$0.285$} & \multicolumn{1}{|r}{$0.058$} & 
\multicolumn{1}{r}{$0.247$} \\ 
\multicolumn{1}{r}{$1000$} & \multicolumn{1}{r}{$800$} & 
\multicolumn{1}{||r}{$0.003$} & \multicolumn{1}{r}{$0.297$} & 
\multicolumn{1}{|r}{$0.029$} & \multicolumn{1}{r}{$0.221$} & 
\multicolumn{1}{|r}{$0.021$} & \multicolumn{1}{r}{$0.189$} \\ 
\multicolumn{1}{r}{$1500$} & \multicolumn{1}{r}{$1199$} & 
\multicolumn{1}{||r}{$0.004$} & \multicolumn{1}{r}{$0.239$} & 
\multicolumn{1}{|r}{$0.020$} & \multicolumn{1}{r}{$0.180$} & 
\multicolumn{1}{|r}{$0.012$} & \multicolumn{1}{r}{$0.157$} \\ 
\multicolumn{1}{r}{$2000$} & \multicolumn{1}{r}{$1599$} & 
\multicolumn{1}{||r}{$0.001$} & \multicolumn{1}{r}{$0.209$} & 
\multicolumn{1}{|r}{$0.013$} & \multicolumn{1}{r}{$0.156$} & 
\multicolumn{1}{|r}{$0.014$} & \multicolumn{1}{r}{$0.121$} \\ \hline\hline
\multicolumn{8}{c}{${\small p=0.9}$} \\ \hline
\multicolumn{1}{r}{$150$} & \multicolumn{1}{r}{$134$} & \multicolumn{1}{||r}{%
$0.034$} & \multicolumn{1}{r}{$0.585$} & \multicolumn{1}{|r}{$0.113$} & 
\multicolumn{1}{r}{$0.479$} & \multicolumn{1}{|r}{$0.118$} & 
\multicolumn{1}{r}{$0.543$} \\ 
\multicolumn{1}{r}{$200$} & \multicolumn{1}{r}{$180$} & \multicolumn{1}{||r}{%
$0.002$} & \multicolumn{1}{r}{$0.512$} & \multicolumn{1}{|r}{$0.120$} & 
\multicolumn{1}{r}{$0.402$} & \multicolumn{1}{|r}{$0.127$} & 
\multicolumn{1}{r}{$0.459$} \\ 
\multicolumn{1}{r}{$300$} & \multicolumn{1}{r}{$270$} & \multicolumn{1}{||r}{%
$0.003$} & \multicolumn{1}{r}{$0.389$} & \multicolumn{1}{|r}{$0.082$} & 
\multicolumn{1}{r}{$0.320$} & \multicolumn{1}{|r}{$0.073$} & 
\multicolumn{1}{r}{$0.310$} \\ 
\multicolumn{1}{r}{$500$} & \multicolumn{1}{r}{$450$} & \multicolumn{1}{||r}{%
$0.002$} & \multicolumn{1}{r}{$0.305$} & \multicolumn{1}{|r}{$0.052$} & 
\multicolumn{1}{r}{$0.246$} & \multicolumn{1}{|r}{$0.045$} & 
\multicolumn{1}{r}{$0.228$} \\ 
\multicolumn{1}{r}{$1000$} & \multicolumn{1}{r}{$900$} & 
\multicolumn{1}{||r}{$0.004$} & \multicolumn{1}{r}{$0.223$} & 
\multicolumn{1}{|r}{$0.024$} & \multicolumn{1}{r}{$0.169$} & 
\multicolumn{1}{|r}{$0.020$} & \multicolumn{1}{r}{$0.153$} \\ 
\multicolumn{1}{r}{$1500$} & \multicolumn{1}{r}{$1349$} & 
\multicolumn{1}{||r}{$0.005$} & \multicolumn{1}{r}{$0.176$} & 
\multicolumn{1}{|r}{$0.020$} & \multicolumn{1}{r}{$0.141$} & 
\multicolumn{1}{|r}{$0.021$} & \multicolumn{1}{r}{$0.124$} \\ 
\multicolumn{1}{r}{$2000$} & \multicolumn{1}{r}{$1800$} & 
\multicolumn{1}{||r}{$0.006$} & \multicolumn{1}{r}{$0.166$} & 
\multicolumn{1}{|r}{$0.013$} & \multicolumn{1}{r}{$0.126$} & 
\multicolumn{1}{|r}{$0.013$} & \multicolumn{1}{r}{$0.110$} \\ \hline
&  &  &  &  &  &  & 
\end{tabular}%
\caption{Biweight-kernel estimation results for the shape parameter
$\gamma_{1}=0.8$ of Burr's model based on 1000 right-truncated samples,
along with other existing estimators}\label{Tab3}%
\end{table}%

\begin{table}[tbp] \centering%
\begin{tabular}{cccccccc}
\hline
\multicolumn{8}{c}{${\small p=0.7}$} \\ \hline
&  & \multicolumn{2}{||c}{$\widehat{\gamma }_{1,\mathbb{K}}$} & 
\multicolumn{2}{|c}{$\widehat{\gamma }_{1}$} & \multicolumn{2}{|c}{$\widehat{%
\gamma }_{1}^{GS}$} \\ \hline
$N$ & $n$ & \multicolumn{1}{||c}{\small abs bias} & {\small rmse} & 
\multicolumn{1}{|c}{\small abs bias} & {\small rmse} & \multicolumn{1}{|c}%
{\small abs bias} & {\small rmse} \\ \hline\hline
\multicolumn{1}{r}{$150$} & \multicolumn{1}{r}{$104$} & \multicolumn{1}{||r}{%
$0.159$} & \multicolumn{1}{r}{$0.976$} & \multicolumn{1}{|r}{$0.202$} & 
\multicolumn{1}{r}{$0.511$} & \multicolumn{1}{|r}{$0.386$} & 
\multicolumn{1}{r}{$3.264$} \\ 
\multicolumn{1}{r}{$200$} & \multicolumn{1}{r}{$139$} & \multicolumn{1}{||r}{%
$0.064$} & \multicolumn{1}{r}{$0.905$} & \multicolumn{1}{|r}{$0.205$} & 
\multicolumn{1}{r}{$0.493$} & \multicolumn{1}{|r}{$0.247$} & 
\multicolumn{1}{r}{$1.355$} \\ 
\multicolumn{1}{r}{$300$} & \multicolumn{1}{r}{$209$} & \multicolumn{1}{||r}{%
$0.090$} & \multicolumn{1}{r}{$0.831$} & \multicolumn{1}{|r}{$0.101$} & 
\multicolumn{1}{r}{$0.469$} & \multicolumn{1}{|r}{$0.141$} & 
\multicolumn{1}{r}{$1.082$} \\ 
\multicolumn{1}{r}{$500$} & \multicolumn{1}{r}{$349$} & \multicolumn{1}{||r}{%
$0.014$} & \multicolumn{1}{r}{$0.589$} & \multicolumn{1}{|r}{$0.090$} & 
\multicolumn{1}{r}{$0.371$} & \multicolumn{1}{|r}{$0.063$} & 
\multicolumn{1}{r}{$0.586$} \\ 
\multicolumn{1}{r}{$1000$} & \multicolumn{1}{r}{$700$} & 
\multicolumn{1}{||r}{$0.013$} & \multicolumn{1}{r}{$0.458$} & 
\multicolumn{1}{|r}{$0.049$} & \multicolumn{1}{r}{$0.296$} & 
\multicolumn{1}{|r}{$0.023$} & \multicolumn{1}{r}{$0.264$} \\ 
\multicolumn{1}{r}{$1500$} & \multicolumn{1}{r}{$1050$} & 
\multicolumn{1}{||r}{$0.008$} & \multicolumn{1}{r}{$0.561$} & 
\multicolumn{1}{|r}{$0.023$} & \multicolumn{1}{r}{$0.315$} & 
\multicolumn{1}{|r}{$0.020$} & \multicolumn{1}{r}{$0.189$} \\ 
\multicolumn{1}{r}{$2000$} & \multicolumn{1}{r}{$1400$} & 
\multicolumn{1}{||r}{$0.012$} & \multicolumn{1}{r}{$0.381$} & 
\multicolumn{1}{|r}{$0.027$} & \multicolumn{1}{r}{$0.241$} & 
\multicolumn{1}{|r}{$0.013$} & \multicolumn{1}{r}{$0.164$} \\ \hline\hline
\multicolumn{8}{c}{${\small p=0.8}$} \\ \hline
\multicolumn{1}{r}{$150$} & \multicolumn{1}{r}{$120$} & \multicolumn{1}{||r}{%
$0.103$} & \multicolumn{1}{r}{$0.886$} & \multicolumn{1}{|r}{$0.151$} & 
\multicolumn{1}{r}{$0.511$} & \multicolumn{1}{|r}{$0.180$} & 
\multicolumn{1}{r}{$1.906$} \\ 
\multicolumn{1}{r}{$200$} & \multicolumn{1}{r}{$160$} & \multicolumn{1}{||r}{%
$0.058$} & \multicolumn{1}{r}{$0.775$} & \multicolumn{1}{|r}{$0.131$} & 
\multicolumn{1}{r}{$0.466$} & \multicolumn{1}{|r}{$0.153$} & 
\multicolumn{1}{r}{$1.311$} \\ 
\multicolumn{1}{r}{$300$} & \multicolumn{1}{r}{$239$} & \multicolumn{1}{||r}{%
$0.023$} & \multicolumn{1}{r}{$0.629$} & \multicolumn{1}{|r}{$0.106$} & 
\multicolumn{1}{r}{$0.398$} & \multicolumn{1}{|r}{$0.078$} & 
\multicolumn{1}{r}{$0.502$} \\ 
\multicolumn{1}{r}{$500$} & \multicolumn{1}{r}{$399$} & \multicolumn{1}{||r}{%
$0.005$} & \multicolumn{1}{r}{$0.515$} & \multicolumn{1}{|r}{$0.069$} & 
\multicolumn{1}{r}{$0.339$} & \multicolumn{1}{|r}{$0.060$} & 
\multicolumn{1}{r}{$0.256$} \\ 
\multicolumn{1}{r}{$1000$} & \multicolumn{1}{r}{$800$} & 
\multicolumn{1}{||r}{$0.005$} & \multicolumn{1}{r}{$0.330$} & 
\multicolumn{1}{|r}{$0.036$} & \multicolumn{1}{r}{$0.226$} & 
\multicolumn{1}{|r}{$0.030$} & \multicolumn{1}{r}{$0.186$} \\ 
\multicolumn{1}{r}{$1500$} & \multicolumn{1}{r}{$1200$} & 
\multicolumn{1}{||r}{$0.017$} & \multicolumn{1}{r}{$0.242$} & 
\multicolumn{1}{|r}{$0.035$} & \multicolumn{1}{r}{$0.176$} & 
\multicolumn{1}{|r}{$0.029$} & \multicolumn{1}{r}{$0.145$} \\ 
\multicolumn{1}{r}{$2000$} & \multicolumn{1}{r}{$1600$} & 
\multicolumn{1}{||r}{$0.001$} & \multicolumn{1}{r}{$0.225$} & 
\multicolumn{1}{|r}{$0.017$} & \multicolumn{1}{r}{$0.160$} & 
\multicolumn{1}{|r}{$0.012$} & \multicolumn{1}{r}{$0.133$} \\ \hline\hline
\multicolumn{8}{c}{${\small p=0.9}$} \\ \hline
\multicolumn{1}{r}{$150$} & \multicolumn{1}{r}{$135$} & \multicolumn{1}{||r}{%
$0.039$} & \multicolumn{1}{r}{$0.611$} & \multicolumn{1}{|r}{$0.117$} & 
\multicolumn{1}{r}{$0.465$} & \multicolumn{1}{|r}{$0.133$} & 
\multicolumn{1}{r}{$1.103$} \\ 
\multicolumn{1}{r}{$200$} & \multicolumn{1}{r}{$180$} & \multicolumn{1}{||r}{%
$0.047$} & \multicolumn{1}{r}{$0.603$} & \multicolumn{1}{|r}{$0.102$} & 
\multicolumn{1}{r}{$0.435$} & \multicolumn{1}{|r}{$0.127$} & 
\multicolumn{1}{r}{$0.845$} \\ 
\multicolumn{1}{r}{$300$} & \multicolumn{1}{r}{$270$} & \multicolumn{1}{||r}{%
$0.020$} & \multicolumn{1}{r}{$0.414$} & \multicolumn{1}{|r}{$0.078$} & 
\multicolumn{1}{r}{$0.308$} & \multicolumn{1}{|r}{$0.071$} & 
\multicolumn{1}{r}{$0.301$} \\ 
\multicolumn{1}{r}{$500$} & \multicolumn{1}{r}{$449$} & \multicolumn{1}{||r}{%
$0.008$} & \multicolumn{1}{r}{$0.321$} & \multicolumn{1}{|r}{$0.049$} & 
\multicolumn{1}{r}{$0.256$} & \multicolumn{1}{|r}{$0.050$} & 
\multicolumn{1}{r}{$0.223$} \\ 
\multicolumn{1}{r}{$1000$} & \multicolumn{1}{r}{$900$} & 
\multicolumn{1}{||r}{$0.011$} & \multicolumn{1}{r}{$0.230$} & 
\multicolumn{1}{|r}{$0.024$} & \multicolumn{1}{r}{$0.173$} & 
\multicolumn{1}{|r}{$0.020$} & \multicolumn{1}{r}{$0.153$} \\ 
\multicolumn{1}{r}{$1500$} & \multicolumn{1}{r}{$1350$} & 
\multicolumn{1}{||r}{$0.008$} & \multicolumn{1}{r}{$0.197$} & 
\multicolumn{1}{|r}{$0.016$} & \multicolumn{1}{r}{$0.137$} & 
\multicolumn{1}{|r}{$0.015$} & \multicolumn{1}{r}{$0.120$} \\ 
\multicolumn{1}{r}{$2000$} & \multicolumn{1}{r}{$1800$} & 
\multicolumn{1}{||r}{$0.001$} & \multicolumn{1}{r}{$0.162$} & 
\multicolumn{1}{|r}{$0.014$} & \multicolumn{1}{r}{$0.115$} & 
\multicolumn{1}{|r}{$0.011$} & \multicolumn{1}{r}{$0.105$} \\ \hline
&  &  &  &  &  &  & 
\end{tabular}%
\caption{Triweight-kernel estimation results for the shape parameter
$\gamma_{1}=0.8$ of Burr's model based on 1000 right-truncated samples,
along with other existing estimators}\label{Tab4}%
\end{table}%

\section{\textbf{Proofs\label{sec4}}}

\noindent The proof is based on a useful weak approximation to the tail
product-limit process recently provided by \cite{BchMN-15b}. From $(\ref%
{limit-1}),$ the estimator $\widehat{\gamma }_{1,\mathbb{K}}$ may be
rewritten into%
\begin{equation*}
\widehat{\gamma }_{1,\mathbb{K}}=\int_{1}^{\infty }x^{-1}\Psi _{\mathbb{K}%
}\left( \frac{\overline{\mathbf{F}}_{n}\left( xX_{n-k:n}\right) }{\overline{%
\mathbf{F}}_{n}\left( X_{n-k:n}\right) }\right) dx.
\end{equation*}%
Recall that $\Psi _{\mathbb{K}}\left( s\right) =s\mathbb{K}\left( s\right) ,$
then it is easy to verify that $\int_{1}^{\infty }x^{-1}\Psi _{\mathbb{K}%
}\left( x^{-1/\gamma _{1}}\right) dx=\gamma _{1}.$ Hence%
\begin{equation*}
\widehat{\gamma }_{1,\mathbb{K}}-\gamma _{1}=\int_{1}^{\infty }x^{-1}\left\{
\Psi _{\mathbb{K}}\left( \frac{\overline{\mathbf{F}}_{n}\left(
xX_{n-k:n}\right) }{\overline{\mathbf{F}}_{n}\left( X_{n-k:n}\right) }%
\right) -\Psi _{\mathbb{K}}\left( x^{-1/\gamma _{1}}\right) \right\} dx.
\end{equation*}%
Let%
\begin{equation}
\mathbf{D}_{n}\left( x\right) :=\sqrt{k}\left( \frac{\overline{\mathbf{F}}%
_{n}\left( xX_{n-k:n}\right) }{\overline{\mathbf{F}}_{n}\left(
X_{n-k:n}\right) }-x^{-1/\gamma _{1}}\right) ,\text{ }x>0,  \label{Dn}
\end{equation}%
be the tail product-limit process, then Taylor's expansion of $\Psi _{%
\mathbb{K}}$ yields that%
\begin{equation*}
\sqrt{k}\left( \widehat{\gamma }_{1,\mathbb{K}}-\gamma _{1}\right)
=\int_{1}^{\infty }x^{-1}\mathbf{D}_{n}\left( x\right) g_{\mathbb{K}}\left(
x^{-1/\gamma _{1}}\right) dx+R_{n1},
\end{equation*}%
with $R_{n1}:=2^{-1}k^{-1/2}\int_{1}^{\infty }x^{-1}\mathbf{D}_{n}^{2}\left(
x\right) g_{\mathbb{K}}^{\prime }\left( \xi _{n}\left( x\right) \right) dx,$
where $\xi _{n}\left( x\right) $\ is a stochastic intermediate value lying
between $\overline{\mathbf{F}}_{n}\left( xX_{n-k:n}\right) /\overline{%
\mathbf{F}}_{n}\left( X_{n-k:n}\right) $ and $x^{-1/\gamma _{1}}.$ According
to \cite{BchMN-15b}, we have, for $0<\epsilon <1/2-\gamma /\gamma _{2}$%
\begin{equation}
\sup_{x\geq 1}x^{\left( 1/2-\epsilon \right) /\gamma -1/\gamma
_{2}}\left\vert \mathbf{D}_{n}\left( x\right) -\mathbf{\Gamma }\left( x;%
\mathbf{W}\right) -x^{-1/\gamma _{1}}\dfrac{x^{\tau _{1}/\gamma _{1}}-1}{%
\gamma _{1}\tau _{1}}\sqrt{k}\mathbf{A}_{0}\left( n/k\right) \right\vert 
\overset{\mathbf{P}}{\rightarrow }0,\text{ as }N\rightarrow \infty ,
\label{approx}
\end{equation}%
where$\mathbb{\ }\left\{ \Gamma \left( x;\mathbf{W}\right) ;\text{ }%
x>0\right\} $\textbf{\ }is a Gaussian process defined by%
\begin{align*}
& \mathbf{\Gamma }\left( x;\mathbf{W}\right) 
\begin{tabular}{l}
$:=$%
\end{tabular}%
\frac{\gamma }{\gamma _{1}}x^{-1/\gamma _{1}}\left\{ x^{1/\gamma }\mathbf{W}%
\left( x^{-1/\gamma }\right) -\mathbf{W}\left( 1\right) \right\}  \\
& \ \ \ \ \ \ +\frac{\gamma }{\gamma _{1}+\gamma _{2}}x^{-1/\gamma
_{1}}\int_{0}^{1}s^{-\gamma /\gamma _{2}-1}\left\{ x^{1/\gamma }\mathbf{W}%
\left( x^{-1/\gamma }s\right) -\mathbf{W}\left( s\right) \right\} ds.
\end{align*}%
Now, we write $\sqrt{k}\left( \widehat{\gamma }_{1,K}-\gamma _{1}\right)
=\int_{1}^{\infty }x^{-1}\mathbf{\Gamma }\left( x;\mathbf{W}\right) g_{%
\mathbb{K}}\left( x^{-1/\gamma _{1}}\right) dx+\sum_{i=1}^{3}R_{ni},$ where%
\begin{equation*}
R_{n2}:=\int_{1}^{\infty }x^{-1}\left\{ \mathbf{D}_{n}\left( x\right) -%
\mathbf{\Gamma }\left( x;\mathbf{W}\right) -x^{-1/\gamma _{1}}\dfrac{x^{\tau
_{1}/\gamma _{1}}-1}{\gamma _{1}\tau _{1}}\sqrt{k}\mathbf{A}_{0}\left(
n/k\right) \right\} g_{\mathbb{K}}\left( x^{-1/\gamma _{1}}\right) dx,
\end{equation*}%
and%
\begin{equation*}
R_{n3}:=\int_{1}^{\infty }x^{-1}\left\{ x^{-1/\gamma _{1}}\dfrac{x^{\tau
_{1}/\gamma _{1}}-1}{\gamma _{1}\tau _{1}}\sqrt{k}\mathbf{A}_{0}\left(
n/k\right) \right\} g_{\mathbb{K}}\left( x^{-1/\gamma _{1}}\right) dx.
\end{equation*}%
Elementary calculation yields that%
\begin{equation*}
\int_{1}^{\infty }x^{-1}\mathbf{\Gamma }\left( x;\mathbf{W}\right) g_{%
\mathbb{K}}\left( x^{-1/\gamma _{1}}\right) dx=\left( \gamma ^{2}/\gamma
_{1}\right) \int_{0}^{1}s^{-1}\mathbf{W}\left( s\right) d\left\{ s\varphi _{%
\mathbb{K}}\left( s\right) \right\} =:Z,
\end{equation*}%
\ where $\varphi _{\mathbb{K}}\left( s\right) $ is that defined in the
theorem. Next, we evaluate the remainder terms $R_{ni},$ $i=1,2,3.$ First,
we show that $R_{n1}$ tends to zero in probability, as $N\rightarrow \infty .
$ Recall that $\gamma _{1}<\gamma _{2}$ and $0<\epsilon <1/2-\gamma /\gamma
_{2},$ then $\left( 1/2-\epsilon \right) /\gamma -1/\gamma _{2}>0.$ It
follows that $\int_{1}^{\infty }x^{2\left( 1/\gamma _{2}-\left( 1/2-\xi
\right) /\gamma \right) -1}dx$ is finite and, from Lemma \ref{Lemma1}, we
get $\sup_{x\geq 1}\left\vert \mathbf{D}_{n}^{2}\left( x\right) \right\vert
=O_{\mathbf{p}}\left( 1\right) .$ On the other hand, from assumption $\left[ 
\mathbb{C}4\right] ,$ we infer that $g_{\mathbb{K}}^{\prime }$ is bounded on 
$\left( 0,1\right) .$ Consequently, we have $R_{n1}=o_{\mathbf{p}}\left(
1\right) .$ Second, for the term $R_{n2},$ we use approximation\ $(\ref%
{approx}),$ to get%
\begin{equation*}
R_{n2}=o_{\mathbf{p}}\left( 1\right) \int_{1}^{\infty }x^{1/\gamma
_{2}-\left( 1/2-\epsilon \right) /\gamma -1}\left\vert g_{\mathbb{K}}\left(
x^{-1/\gamma _{1}}\right) \right\vert dx.
\end{equation*}%
Since $g_{\mathbb{K}}$ is bounded on $\left( 0,1\right) ,$ then $R_{n2}=o_{%
\mathbf{p}}\left( 1\right) .$ Finally, we show that the third term $R_{n3}$
is equal to $\sqrt{k}\mathbf{A}_{0}\left( n/k\right) \int_{0}^{1}s^{-\tau
_{1}}\mathbb{K}\left( s\right) ds.$ Observe that%
\begin{equation*}
R_{n3}=\sqrt{k}\mathbf{A}_{0}\left( n/k\right) \int_{1}^{\infty
}x^{-1/\gamma _{1}-1}\dfrac{x^{\tau _{1}/\gamma _{1}}-1}{\gamma _{1}\tau _{1}%
}g_{\mathbb{K}}\left( x^{-1/\gamma _{1}}\right) dx.
\end{equation*}%
By using a change of variables and by replacing $\Psi _{\mathbb{K}}\left(
s\right) =s\mathbb{K}\left( s\right) ,$ we end up with%
\begin{equation*}
\int_{1}^{\infty }x^{-1/\gamma _{1}-1}\dfrac{x^{\tau _{1}/\gamma _{1}}-1}{%
\gamma _{1}\tau _{1}}g_{\mathbb{K}}\left( x^{-1/\gamma _{1}}\right)
dx=\int_{0}^{1}s^{-\tau _{1}}\mathbb{K}\left( s\right) ds.
\end{equation*}%
For the second part of the theorem, it suffices to use Lemma 8 in \cite{CDM}%
, to show that the variance of the centred Gaussian rv $Z$ equals $\sigma _{%
\mathbb{K}}^{2}.$ Finally, whenever $\sqrt{k_{N}}\mathbf{A}_{0}\left(
N/k_{N}\right) \rightarrow \lambda ,$ we have $R_{n3}\overset{\mathbf{p}}{%
\rightarrow }\lambda \int_{0}^{1}s^{-\tau _{1}}\mathbb{K}\left( s\right) ds,$
as $N\rightarrow \infty ,$ which corresponds to the asymptotic bias $\mu _{%
\mathbb{K}}$ as sought.

\section{Appendix}

\begin{lemma}
\label{Lemma1}Under the assumptions of Theorem \ref{Theorem1}, we have, for
any $0<\epsilon <1/2-\gamma /\gamma _{2}$%
\begin{equation*}
\sup_{x\geq 1}x^{\left( 1/2-\epsilon \right) /\gamma -1/\gamma
_{2}}\left\vert \mathbf{D}_{n}\left( x\right) \right\vert =O_{\mathbf{p}%
}\left( 1\right) ,\text{ as }N\rightarrow \infty .
\end{equation*}
\end{lemma}

\begin{proof}
This result is straightforward from the weak approximation\ $(\ref{approx}).$
Indeed, it is clear that $\sup_{x\geq 1}x^{\left( 1/2-\epsilon \right)
/\gamma -1/\gamma _{2}}\left\vert \mathbf{D}_{n}\left( x\right) \right\vert
\leq T_{1,n}+T_{2,n}+T_{3},$ where%
\begin{equation*}
T_{1,n}:=\sup_{x\geq 1}x^{\left( 1/2-\epsilon \right) /\gamma -1/\gamma
_{2}}\left\vert \mathbf{D}_{n}\left( x\right) -\mathbf{\Gamma }\left( x;%
\mathbf{W}\right) -x^{-1/\gamma _{1}}\dfrac{x^{\tau _{1}/\gamma _{1}}-1}{%
\gamma _{1}\tau _{1}}\sqrt{k}\mathbf{A}_{0}\left( n/k\right) \right\vert ,
\end{equation*}%
\begin{equation*}
T_{2,n}:=\frac{\sqrt{k}\mathbf{A}_{0}\left( n/k\right) }{\gamma _{1}\tau _{1}%
}\sup_{x\geq 1}\left\{ x^{-\left( 1/2+\epsilon \right) /\gamma }\left(
1-x^{\tau _{1}/\gamma _{1}}\right) \right\} \text{ and }T_{3}:=\sup_{x\geq
1}x^{\left( 1/2-\epsilon \right) /\gamma -1/\gamma _{2}}\left\vert \mathbf{%
\Gamma }\left( x;\mathbf{W}\right) \right\vert .
\end{equation*}%
First, it is readily checked from $(\ref{approx})$ that $T_{1,n}=o_{\mathbf{p%
}}\left( 1\right) .$ Second, observe that, in addition to the assumption $%
\sqrt{k}\mathbf{A}_{0}\left( n/k\right) =O_{\mathbf{p}}\left( 1\right) ,$ we
have $x^{-\left( 1/2+\epsilon \right) /\gamma }\left( 1-x^{\tau _{1}/\gamma
_{1}}\right) \leq 2,$ for $x\geq 1,$ it follows that $T_{2,n}=O_{\mathbf{p}%
}\left( 1\right) .$ Finally, note that $x^{\left( 1/2-\epsilon \right)
/\gamma -1/\gamma _{2}}\mathbf{\Gamma }\left( x;\mathbf{W}\right) $ is equal
to%
\begin{align*}
& x^{-\left( 1/2+\epsilon \right) /\gamma }\left\{ \frac{\gamma }{\gamma _{1}%
}\left( x^{1/\gamma }\mathbf{W}\left( x^{-1/\gamma }\right) -\mathbf{W}%
\left( 1\right) \right) \right. \\
& \text{ \ \ \ \ \ \ \ \ \ \ \ }+\left. \frac{\gamma }{\gamma _{1}+\gamma
_{2}}\int_{0}^{1}s^{-\gamma /\gamma _{2}-1}\left( x^{1/\gamma }\mathbf{W}%
\left( x^{-1/\gamma }s\right) -\mathbf{W}\left( s\right) \right) ds\right\} ,
\end{align*}%
where the quantity between brackets is a Gaussian rv and $x^{-\left(
1/2+\epsilon \right) /\gamma }\leq 1,$ for $x\geq 1.$ Therefore, $T_{3}=O_{%
\mathbf{p}}\left( 1\right) $ and the proof is completed.
\end{proof}

\end{document}